\documentclass[12pt,a4paper]{amsart}
\usepackage{color}
\usepackage{amssymb}
\usepackage{graphicx}

\usepackage{psfrag}

\def\Z{\mathbb{Z}}

\def\e{\varepsilon}

\def\G{\Gamma}

\def\<{\langle}
\def\>{\rangle}
\def\-{\overline}

\def\deff{{\rm{def}}}

\newtheorem{theorem}{Theorem}[section]
\newtheorem{lemma}[theorem]{Lemma}

\newtheorem{proposition}[theorem]{Proposition}

\newtheorem{corollary}[theorem]{Corollary}
\newtheorem{problem}[theorem]{Problem}

\theoremstyle{definition}

\newtheorem{remarks}[theorem]{Remarks}
\newtheorem{remark}[theorem]{Remark}

\numberwithin{equation}{section}

\catcode`\@=11
\def\serieslogo@{\relax}
\def\@setcopyright{\relax}
\catcode`\@=12

\begin{document}

\title[Concise Presentations of Direct Products]{Concise Presentations of Direct Products}
\thanks{The author is supported by a Wolfson Research Merit Award from the Royal Society}

\author[Bridson]{Martin R.~Bridson}
\address{Martin R.~Bridson\\
Mathematical Institute \\
Andrew Wiles Building\\
Oxford OX2 6GG \\ 
European Union} 
\email{bridson@maths.ox.ac.uk}



\begin{abstract}Direct powers of perfect groups admit more concise presentations than one might naively suppose.
If $H_1(G,\Z)=H_2(G,\Z)=0$, then $G^n$ has a presentation with $O(\log n)$ generators and $O(\log n)^3$ relators.
If, in addition, there is an element $g\in G$ that has infinite order in every non-trivial quotient of $G$,
then $G^n$ has a presentation with $d(G) +1$ generators and $O(\log n)$ relators. 
The bounds that we obtain on the deficiency of $G^n$ are not monotone in $n$; this points to potential counterexamples for the Relation Gap Problem.
\end{abstract} 

\maketitle

\section{Introduction}

If two groups are presented as $A=\<X\mid R\>$ and  
$B= \<Y\mid S\>$, then their direct product is given by the presentation 
with generators $X\sqcup Y$ and relators $ R,\, S$ and $\{ [x,y] : x\in X,\, y\in Y\}$.
Similarly, if $A_i=\< X_i \mid R_i\>$ with $|X_i|=k_i$ and $|R_i|=l_i$, then the obvious
presentation of $A_1\times \dots \times A_n$ has $\sum k_i$ generators and $\sum l_i + \sum_{i< j}k_ik_j$
relators. In particular, the direct product $A^n$ of $n$ copies of $A=\<X\mid R\>$ with $|X|=k$ and $|R|=l$
has a presentation with $kn$ generators and $nl+ k^2n(n-1)/2$ relators. 
In the absence of further hypotheses, one cannot do better than these naive bounds. For example,
 one cannot generate $\Z^n$ with fewer than $n$ generators, and the number of relators needed to 
present $\Z^n$ is at least the rank of $H_2(\Z^n,\Z) \cong \Z^n\wedge \Z^n$,
which is $n(n-1)/2$. But when $H_1(G,\Z)$ and $H_2(G,\Z)$ vanish, one can construct much more concise
presentations of $G^n$ -- that is the main theme of this note. 
 
 We shall see that,
in addition to the vanishing of homology, 
the existence of finite quotients of $G$ plays a key role
in determining how concise a presentation of $G^n$ can be.
 The various possibilities are summarised in the following theorem, in which we use the standard
 notation $d(\G)$ for the minimal number of generators of $\G$ and we define
 $\rho(\G)$ to be the minimum number of relators in any finite presentation of $\G$.
 All of the results concerning the growth of $d(G^n)$ are taken from \cite{WW}; they
draw on earlier results of Hall \cite{hall}, Wiegold \cite{W1, W2, W3} and others. 
 The estimates on $\rho(G^n)$ are new (or trivial).
 
 We use the standard notation $f(n) = \Theta(g(n))$ for functions that are bounded 
 above and below by positive multiples of $g(n)$, and for brevity we write $H_iG$ in place of $H_i(G,\Z)$.
 Throughout, $G^n$ denotes  the direct product of $n$ copies of $G$.

\begin{theorem}\label{t:summary} Let $G$ be a finitely presented group.
\smallskip

\noindent {\rm{(1)}} If $H_1G\neq 0$, then 
$d(G^n) = \Theta(n)$ and $\rho(G^n)=\Theta(n^2)$.
\smallskip

\noindent {\rm{(2)}} If $H_1G=0$ and $H_2G\neq 0$, then $d(G^n)=O(\log n)$ and  $\rho(G^n) = \Theta(n)$.
\smallskip

\noindent {\rm{(3)}} If $H_1G=H_2G=0$, then $d(G^n)=O(\log n)$ and $\rho(G^n)=O(\log n)^3$.
\smallskip

\noindent {\rm{(4)}} If $H_1G=H_2G=0$ and $G$ has a non-trivial finite quotient, then $d(G^n)=\Theta(\log n)$ and there are constants $c_0,c_1$ such that
$c_0 \log n \le \rho(G^n) \le c_1(\log n)^3$.
\smallskip

\noindent {\rm{(5)}}  If $H_1G=H_2G=0$ and there is an element $g\in G$ that has infinite order in every non-trivial quotient of $G$,
then $d(G^n)\le d(G) +1$ for all $n$, and $\rho(G^n) = O(\log n)$. 

\medskip
\noindent In all cases, the upper bounds on $d(G^n)$ and $\rho(G^n)$ can be satisfied simultaneously.
\end{theorem}

I see no reason to expect that $\rho(G^n)$ is a monotone function of $n$ for all finitely presented perfect groups $G$, and this is intriguing in the context of the celebrated
{\em{Relation Gap Problem}} \cite{jens}.  
Recall that the {\em deficiency} of a finite group presentation 
$\< A \mid R\>$ is $|R|-|A|$, and the deficiency $\deff(G)$ 
of a group $G$ is
defined\footnote{there are two conventions in the literature: many authors take this
definition to be $-\deff(G)$.} to be the least deficiency among all finite presentations of $G$. Our constructions suggest that the following problem might have a positive answer. If it does, then  $\Gamma^m$ would be a counterexample to the  
Relation Gap Problem: see remark \ref{r:rel-gap} for an explanation and variations.

\begin{problem} Does there exist a finitely presented perfect group $\G$ and a
positive integer $m$ such that $\deff(\Gamma^m)>\deff(\Gamma^{m+1})$ 
or $\rho(\Gamma^m)>\rho(\Gamma^{m+1})$?
\end{problem}

\section{Proofs}

We shall need some basic facts about universal central extensions of groups. 

A {\em central extension} of a group $G$ is
a group $\widetilde G$ equipped with an epimorphism $\pi:\widetilde{G}\to G$ whose kernel is central
in $\widetilde G$. Such an extension is {\em universal} if given any other central extension
$\pi': E \to G$ of $G$, there is a unique homomorphism $f : \widetilde G \to  E$ such that
$\pi'\circ f = \pi$.  The
standard reference for this material is \cite{milnor} pp. 43--47.   The
properties that we need here are these: $G$ has a universal central extension $\widetilde G$
if (and only if) $H_1(G, \Z) = 0$; there is a short exact sequence
$$1 \to  H_2(G,\Z) \to \widetilde{G}\to G\to 1\ ;$$
and if $G$ has no non-trivial finite quotients then neither does $\widetilde G$.

The following result is Proposition 3.5 of \cite{mb:karl}.

\begin{lemma}\label{l:Phat} Let $G=\<X\mid R\>$ be a perfect group,
let $F$ be the free group on $X$ and for each $x\in X$ let $c_x\in [F,F]$ be a word
such that $x=c_x$ in $G$. Then the following is a presentation of the universal central extension
of $G$:
\begin{equation}\label{Ghat}
\tilde G = \< X\mid x^{-1}c_x,\, [r,x]\ (\forall r\in R,\, x\in X)\>,
\end{equation}
and the identity map $X\to X$ extends uniquely to an epimorphism $\tilde G\to G$ with kernel isomorphic to $H_2(G,\Z)$.
\end{lemma}


\begin{corollary} If $H_1(G,\Z)=H_2(G,\Z)=0$, then (\ref{Ghat})
is a presentation of $G$. 
\end{corollary}

It will be convenient to use functional notation for words.
Thus, given a word $u$ in the symbols $x_1^{\pm 1},\dots,x_k^{\pm 1}$, we write $u(\underline{x})$ to emphasize
the underlying alphabet and we write $u(\underline{y})$ for the word obtained by replacing each occurrence
of each $x_i$ with $y_i$, where $y_1^{\pm 1},\dots,y_k^{\pm 1}$ is a second (ordered) alphabet.

\begin{proposition}\label{p:pres}
Let $G=\<x_1,\dots,x_k\mid r_1,\dots,r_l\>$, let $F$ be the
free group on the $x_i$, suppose that $H_1(G,\Z)=H_2(G,\Z)=0$, and for each $x_i$ fix $c_i(\underline{x})\in[F,F]$
such that $x_i=c_i(\underline{x})$ in $G$. Then the following is a presentation of $G\times G$:
\begin{equation}\label{presGxG}
\< x_1,\dots,x_k, y_1,\dots, y_k \mid r_1,\dots, r_l,\, y_i^{-1}c_i(\underline{y}),\,
[x_iy_i^{-1},y_j],\, 1\le i,j\le k\>.
\end{equation}
\end{proposition}

\begin{proof}  First observe that the  
last family of relations
can be written as $x_i^{-1}y_jx_i=y_i^{-1}y_jy_i$, from which it follows that
$x_i^{-1}ux_i= y_i^{-1}uy_i$ for all words $u$ in the free group on $\{y_1,\dots,y_k\}$ and each
$i=1,\dots,k$. Therefore, in the group presented,
the transcription $r_j(\underline{y})$ of each relation $r_j(\underline{x})$
is central in the subgroup $G_1:=\<y_1,\dots,y_k\>$, because $y_i=y_i^{r_j(\underline{x})}
=y_i^{r_j(\underline{y})}$. Thus $G_1$ (which is clearly
normal) satisfies the relations that were used in
 Lemma \ref{l:Phat} to define the universal central extension $\widetilde{G}$. And
 $\widetilde{G}=G$, because $H_2(G,\Z)=0$.
 
At this stage we know that the group given by presentation (\ref{presGxG}) has the form $G_1\rtimes G_2$,
with $G_1\cong G_2\cong G$, where $G_1$ is the subgroup generated by the $y_i$ and $G_2$ is the subgroup
generated by the $x_i$. The action $\phi:G_2\to {\rm{Aut}}(G_1)$ defining the semidirect product is by inner
automorphisms, $x_i\mapsto {\rm{ad}}_{y_i}$. Because this action factors $G_2\to G_1\to {\rm{Inn}}(G_1)$,
we have $G_1\rtimes G_2 \cong G_1\times G_2$; indeed an isomorphism 
$\phi:G_1\times G_2 \to G_1\rtimes G_2$ is given by $\phi(y_i)=y_i$ and 
$\phi(x_i)= y_i^{-1}x_i$.
\end{proof}

At first blush, this proposition seems to gain us little or nothing compared to the naive presentation of $G\times G$:
we have traded the $l$ obvious relations of $G_1$ for the $k$ relations $y_i^{-1}c_i(\underline{y})$. 
The real benefit comes when we iterate the construction and use the fact that the number of generators
that $G^n$ requires grows strikingly slowly (an old observation of Philip Hall \cite{hall}). To exploit this
we need: 

\begin{lemma}\label{l:const} Let $G=\<x_1,\dots,x_k\mid r_1,\dots,r_l\>$ and suppose $H_1G=H_2G=0$. 
If $G^N$ requires at
most $k$ generators and $2^n\le N$, then $G^{2^n}$ has a presentation with 
$k$ generators and $n(k^2+2k)+l$ relations. 
\end{lemma}

\begin{proof} As in the previous proof, we construct a presentation of $G^2$ with 
$2k$ generators $b_1,\dots,b_{2k}$ and $k^2+k+l$ relations. We then make Tietze moves to add a new generating set
$a_1,\dots,a_k$, together with $k$ relations expressing the $a_i$ as words in the generators $b_i$.
There are words $u_i$ in the generators $a_j$ such that $b_i=u_i$ in $G\times G$. We make
further Tietze moves, removing the generators $b_i$ and replacing each occurrence of $b_i$ in the
relators by $u_i$. Thus we obtain a presentation of $G\times G$ with 
$k$ generators and $k^2+2k+l$ relators.

Repeating the argument with $G\times G$ in place of $G$, we obtain a presentation for $G^4$ with 
$k$ generators and $2(k^2+2k)+l$ relators. And continuing in this manner (provided that we stay
in the range where $G^{2^n}$ needs only $k$ generators), we obtain a presentation for $G^{2^n}$
with $k$ generators and $n(k^2+2k)+l$ relators.
\end{proof}

\begin{corollary}\label{c:const} If $G$ and $N$ are as in the lemma and $m\le N/2$, then $G^m$ has a
presentation with $k$ generators and $(k^2+2k)(\log_2 m +1) + l +k$ relators.
\end{corollary}

\begin{proof} Let $n$ be the least integer such that $m\le 2^n$ and write $G^{2^n} = G^m\times G^{2^n-m}$.
The lemma tells us that $G^{2^n}$ has a presentation with $k$ generators and $n(k^2+2k)+l$ relators.
Moreover, as $2^n-m<N$, the second factor in the given decomposition is a $k$-generator group, and can therefore be
killed by the addition of at most $k$ relations. To complete the proof, note that $n-1< \log_2 m$.
\end{proof}

It is an open question as to whether every finitely generated perfect group is the
normal closure of one element. If it is, then the $k$ relations added to kill $G^{2^n-m}$ in the above
proof could be replaced by a single relation.  

\smallskip
We shall need the following result of
Wiegold and Wilson \cite{WW}; the proof presented here is new but has much in common with the original. 
\begin{proposition}\label{l:d-is-log}
Let $G$ be a perfect group. If $d(G)=r$, then $d(G^m) \le r(1+ \lceil \log_2 (m+1)\rceil)$.
\end{proposition}

\begin{proof} Let $M=\lceil\log_2 (m+1)\rceil$, the least integer with $m<2^M$.
The proof uses binary expansions $j = \sum_{i=0}^{M-1} \e_i(j)2^i$ of integers
$j=1,\dots,m$.
Given a generating set $\{a_1,\dots,a_k\}$ for $G$, for $i=0,\dots,M-1$ we define
$$a_{r,i} = (a_r^{\e_i(1)}, a_r^{\e_i(2)}, \dots , a_r^{\e_i(m)}).$$
For each pair of integers $1\le j<j' \le m$, there is some $i$ such that 
$\e_i(j)\neq \e_i(j')$, and for that $i$ we have $p_{j,j'}(a_{r,i})\in\{(a_r,1),
\, (1,a_r)\}$, where $p_{jj'}:G^m\to G\times G$ is the coordinate projection to the $j$ and $j'$ factors. The image under $p_{jj'}$ of the diagonal element $\alpha_r:=(a_r,\dots,a_r)$ is $(a_r,a_r)$. Thus the restriction of $p_{jj'}$ to the subgroup $S<G^n$ generated by the set
$
\{\alpha_r,\ a_{r,i} \mid r=1,\dots,k;\, i=0,\dots,M-1\}
$
is surjective. It follows   
that $S$ contains the $(m-1)$-st term of the lower central
series of $G^m$ (see \cite{BMi} p.643). But $G^m$ is perfect, so each term of the
lower central series is the whole group, and therefore $S=G^m$. 
\end{proof}
 
\begin{theorem}\label{t:easy} If $G$ is a finitely presented group with $H_1G=H_2G=0$, then
$G^m$ has a finite presentation with at most $O(\log m)$ generators and $O(\log m)^3$ relators.
\end{theorem}

\begin{proof} The preceding proposition shows that $d(G^m) = O(\log m)$.  
We fix a constant $k$ so that $G$ can be generated by $k$ elements and
$G^{m}$ can be generated by $k\lceil \log_2 m\rceil$ elements, 
for all positive integers $m$. Suppose $G=\<x_1,\dots,x_k\mid r_1,\dots,r_l\>$. Since $G^2$ only needs $k$
generators, 
as in Lemma \ref{l:const}
we obtain a presentation of $G^2$ with 
$k$ generators and $k^2+2k+l$ relators. From Proposition \ref{p:pres} we then 
get a presentation of $G^4=G^2\times G^2$ with $2k$ generators and $k^2+k+(k^2+2k+l)=2k^2+3k+l$ relators.
Applying Proposition \ref{p:pres} again
we get a presentation of $G^8$ with $4k$ generators and $(2k)^2+2k+(2k^2+3k+l)=6k^2+5k+l$ relators.
Since $G^8$ only requires $3k$ generators,  as in the proof of Lemma 2.4 we can convert this to a presentation 
with $3k$ generators and $6k^2+8k+l$ relators. 

Repeating this argument, we obtain a presentation of $G^{16}$ with $6k$ generators and
$(3k)^2+3k+(6k^2+8k+l)=15k^2+11k+l$ relators, which we convert to
one with $4k$ generators and $15k^2+15k+l$ relators. And, proceeding by induction,
we get a presentation of $G^{2^n}$ with $nk$ generators and $\sigma_n k^2 + \tau_n k + l$
relators, where $\sigma_n -1 = n(n-1)(2n-1)/6$ is the sum of squares up to $(n-1)^2$
and $\tau_n = n^2-1$. 

Given $m$, we let $n=\lceil \log_2 m\rceil$, write $G^{2^n} = G^m\times G^{2^n-m}$,
take the presentation of $G^{2^n}$ constructed above and kill the factor $G^{2^n-m}$
by adding relations to kill a generating set of cardinality $k\lceil \log_2 (2^n - m)\rceil $,
which is at most $k(n-1)$. Thus we obtain a presentation of $G^m$ with $kn = O(\log m)$ generators and 
$\sigma_n k^2 + \tau_n k + l + k(n-1) = O(\log m)^3$ relators.
\end{proof}

\noindent{\bf{Proof of Theorem \ref{t:summary}}} All of the results 
that we need concerning the growth of $d(G^n)$ can be found in \cite{WW}; they
draw on earlier results of Hall \cite{hall}, Wiegold \cite{W1, W2, W3} and others. Thus we focus on 
the estimates for $\rho(G^n)$.

A simple induction using the K\"unneth formula shows that 
if $H_1G\neq 0$ then  the number of generators needed for
$H_2G^n$ is at least $n(n-1)/2$, so one needs at least this
number of relations to present $G^n$.
The complementary upper bound is provided by the naive construction in the first paragraph of the Introduction. This proves (1).

If $G$ is perfect, then by the K\"unneth formula $H_2G^n$ is a direct sum of $n$ copies of $H_2G$, and therefore
$d(H_2G^n)$ grows linearly if $H_2G\neq 0$. This provides the lower bound for (2).
To establish a complementary upper bound, we consider the
universal central extension $\widetilde G$. Theorem \ref{t:easy} tells us that 
$\widetilde G^n$ has a presentation with at most $O(\log_2n)$ generators and $O(\log n)^3$ relations. The kernel of 
$\widetilde G\to G$ is isomorphic to $H_2G$,
so we need only add a further $n\, d(H_2G)$ relations to pass from $\widetilde{G}^n$  to the quotient $G^n$. 

(3) is Theorem \ref{t:easy}. The bounds on the number of relations in (4) follow from (3) and the
simple observation that since $H_1G^n=0$, the number of relators in any presentation is at least as great as
the number of generators. 

If $G$ is perfect and $g\in G$ has infinite order in every non-trivial quotient of $G$, then $G^n$ is generated by the diagonal copy of $G$ together with $(g,g^2,\dots,g^n)$, by
Theorem 4.4 of \cite{WW}; hence $d(G^n)\le d(G)+1$, as asserted in (5).
The required bound on $\rho(G^m)$ is a special case of Corollary \ref{c:const}.
\qed

\smallskip

\noindent{\bf{Relation Gap Problem.}}

\begin{remark} \label{r:rel-gap}
If one expresses a finitely presented 
group $G$ as a quotient of a free group $G\cong F/R$,
then the action of $F$ by conjugation on $M=R/[R,R]$ makes $M$ a $\Z F$-module (and a $\Z G$-module). It is obvious
that this module requires at most $d_F(R)$ generators, where $d_F(R)$ is the least number of elements (relations of $G$) that one needs to generate $R$ as a normal subgroup of $F$. Despite much effort, there is no example known where $M$ is proved to
require fewer than $d_F(R)$ generators -- the putative difference is the {\em relation gap}. 

An elementary calculation shows that if $N<G$ is normal and perfect, then the relation module for $F\to G/N$
requires no more generators than $M$ does, but one suspects that in some cases $G/N$ is finitely
presented and requires more relations
than $G$. 
For example, if one could prove that there is a finitely presented perfect group $\G$ such
that $\deff(\G^n) > \deff(\G^m)$ for some $n<m$, then one could take a finite presentation realising the deficiency of $\G^m$ and add relations to kill a direct factor $\G^{m-n}$;
the resulting presentation of $\G^n$ would have a relation gap of at least
$\deff(\G^n)-\deff(\G^m)$. 

Similarly, if $\rho(\G^n)>\rho(\G^m)$ for some $n<m$,
then by taking a presentation of $\G^m$ with $\rho(\G_m)$ relators
and passing to $\G^n$ by killing a direct factor $\G^{m-n}$,
we would obtain a presentation with a relation gap. More generally, it would suffice
to prove that a specific map $F\to\G^n$ from a finitely generated free group factored
as $F\to \G^m\to \G^n$, where the second map is the quotient by a direct factor 
and the kernel of the first map requires fewer normal
generators than the composite. The special role that powers of the form
$G^{2^r}$ play in the proofs of this section is intriguing in this regard. 
\end{remark}

\section{Examples}

\subsection{Profinitely trivial examples}

In \cite{BG} Fritz Grunewald and I
constructed a family of infinite super-perfect groups $B_p$ that have no non-trivial
finite quotients. The presentation given there is
$$ 
B_p = \< a, b, \alpha, \beta \mid ba^pb^{-1}=a^{p+1},\, \beta\alpha^p\beta^{-1}=\alpha^{p+1},\, 
[bab^{-1},a]\beta^{-1},\, [\beta\alpha\beta^{-1},\alpha]b^{-1}\>.
$$ 
A 3-generator, 3-relator presentation of $B_p$ 
can be obtained from this by a simple Tietze move removing the generator $\beta$ and
the third relation, replacing the occurrences of $\beta$ in the second and fourth
relations by the word  $[bab^{-1},a]$.

\begin{lemma} Let $Q$ be a quotient of $H=\<a,b\mid ba^pb^{-1}=a^{p+1}\>$. If the image of $a$ in $Q$
has finite order, then the image of $[bab^{-1},a]$ is trivial.
\end{lemma}

\begin{proof} If the image $\overline a$ of $a$ has finite order, then the images of $a^p$ and $a^{p+1}$  in $Q$
must have the same order, since they are conjugate. But the order of $\overline a^r$ is $m/c$, where $m$
is the order of $\overline a$ and $c=(m,r)$ is the highest common factor. Since $p$ and $p+1$ are coprime,
it follows that  $\overline{a}^p$ generates $A=\<\overline a\>$ and the image of $b$ conjugates
$\overline a$ to a power of $\overline a$. In particular, the image of $[bab^{-1},a]$ in $Q$ is trivial.
\end{proof}

We need the following strengthening of the fact that $B_p$ has no non-trivial finite quotients.

\begin{proposition} $a\in B_p$ has infinite order in every non-trivial quotient of $B_p$.
\end{proposition}

\begin{proof} If the image of $a$ has finite order in a quotient $Q$, 
then the image of $[bab^{-1},a]$ is trivial, by the lemma.
The relations $\beta = [bab^{-1},a]$ and $\beta\alpha^p\beta^{-1}
 = \alpha^{p+1}$ then force $\beta$ and
$\alpha$ to have trivial image in $Q$, whence $b = [\beta\alpha\beta^{-1}, \alpha]$ does too. So $Q=1$.
\end{proof}

\begin{theorem}
For all integers $p,m$, the direct product of $m$ copies of $B_p$ has a presentation with
at most $4$ generators and $24\lceil \log_2 m \rceil -1$ relations.
\end{theorem}

\begin{proof} By Theorem \ref{t:summary}(5) (which is from \cite{WW}) or Remark \ref{bm}(1),
we know that $B_p^n$ requires at most 4 generators. To estimate the
number of relations needed, first,
as in Proposition \ref{p:pres}, we present $B_p\times B_p$ with
$6$ generators and $15$ relations. Then, 
as in the proof of Lemma \ref{l:const}, we reduce this to a presentation
with $4$ generators and $19$ relations. Continuing the argument of Lemma \ref{l:const}, we get a 4-generator presentation
of $B_p^4$ with   $16 + 4 +19 + 4 = 43$ relations, then a 4-generator presentation of $B_p^8$ with
$16+4+43+4= 67$ relations, a 4-generator presentation of $B_p^{16}$ with
$16+4+67+4=91$ relations, and a 4-generator presentation of $B_p^{2^n}$ with
$24n-5$ relations. As in Lemma \ref{l:const}, we conclude that $B_p^m$ has a 4-generator presentation with at most
$24\lceil \log_2 m \rceil -1$ relations.
\end{proof}
 
\begin{remarks}\label{bm}
(1) I do not know if the number of relations needed to present $B_p^m$ is $\Omega(\log m)$.

(2) In \cite{BM}, Baumslag and Miller constructed a 4-generator finitely presented group $G_p$ that admits a surjection 
$G_p\to G_p\times G_p$. The group $B_p$ is a quotient of $G_p$, and therefore $B_p^n$ is a quotient of $G_p$ for
all positive integers $n$. 
\end{remarks}
 
\subsection{Infinite simple groups}

The Burger-Mozes groups are infinite simple groups that arise as the fundamental groups of compact non-positively squared 2-complexes
\cite{BuM}. Such a complex $X$
is a classifying space for its fundamental group $\G=\pi_1X$, so $H_2\G = H_2X$. These complexes have many more 
2-cells than 1-cells, so $H_2X$ is a free-abelian group of non-zero rank. By combining parts (2) and (5) of Theorem \ref{t:summary}, we see that
$\G^n$ has a finite presentation with at most $d(\G)+1$ generators but the number of relations needed to present $\G^n$ grows linearly.

Rattaggi \cite{ratt} refined the original construction of Burger and Mozes to produce examples with relatively small presentations.
In particular he constructed an example
with $3$ generators and $62$ relations. 

Richard Thompson's group $T$ provides a further example of a
3-generator infinite simple group \cite{MT} (see \cite{cannon}, for example). 
Ghys and Sergiescu \cite{GS} proved that $H_2(T,\Z)\neq 0$, so again $T^n$ needs at most 4 generators
but the number of relations required to present $T^n$ grows linearly with $n$.

\subsection{Finite groups} 

Super-perfect finite groups are covered by Theorem \ref{t:summary}(4). It would be particularly interesting to
improve the estimate $\rho(G^n) = O(\log n)^3$ in this case, where one has so much more structure. 

To close, we follow our construction in the case of the binary icosahedral group $\widetilde A_5\cong{\rm{SL}}(2,5)$.
Since it is the universal central extension of $A_5$, we have $d(\widetilde A_5^n) = d(A_5^n)$. Famously,
Philip Hall \cite{hall} calculated the range of $n$ in which $d(A_5^n)=2,3$.
By following our argument in this case we get,
for example, that $\widetilde A_5^{16}$ has a 2-generator presentation with $36$ relators, while
$\widetilde A_5^{1024}$ has a 3-generator presentation with $118$ relators.

\smallskip
\noindent{\bf Acknowledgment.}
The exposition in this article benefited from the perceptive comments of a diligent referee.

\end{document}